\documentclass[11pt]{amsart}

\newtheorem{theorem}{Theorem}[section]
\newtheorem{example}{Example}[section]
\newtheorem{lemma}{Lemma}[section]
\newtheorem{remark}{Remark}[section]

\numberwithin{equation}{section}

\newcommand{\R}{\mathbb{R}}

\title[degenerate Stochastic differential equations]
{A Class of degenerate Stochastic differential equations with non-Lipschitz
coefficients}
\author{K. Suresh Kumar }
\address{Department of Mathematics, Indian Institute of Technology Bombay,
Mumbai-400076, India.}
\email{suresh@math.iitb.ac.in}

\date{}

\begin{document}
\abstract
We obtain sufficient condition for SDEs to evolve in the positive orthrant. 
We use comparison theorem arguments to achieve this. As a result 
we prove the existence of a unique strong solution for a class of multidimensional 
degenerate SDEs with non-Lipschitz diffusion coefficients. 
\endabstract


\maketitle

\section{Introduction}
In this article we consider (possibly degenerate) stochastic differential equations 
(SDEs) with non-Lipshitz coefficients. If the coefficients are Lipschitz, we can prove 
the existence of a unique strong solution (see \cite{nobu}). But uniqueness fails in the 
case of non-Lipschitz coefficients.  The literature on this topic is not very extensive. 
However, in one dimensional case, there is an extensive literature (see \cite{chen}). 

If the coefficients are continuous, then weak solution exists upto an explosion time 
(see \cite{nobu}, pp. 155 - 163).  Under additional linear
growth conditions,  weak solution exists for all time. 
In the case of non-degenerate diffusion coefficient, Stroock and Varadhan proved the
existence of a unique weak solution for SDE with  bounded and mesurable drift and with 
bounded and continous
diffusion coefficient (see \cite{sv} Theorem 4.2, 5.6).
Krylov  \cite{kry}  relaxed the continuity assumption on diffusion coefficient and proved
that there is a unique weak solution  for $n \leq 2$, and for $n > 2$,  the sde has a 
weak solution.  Engelbert and Schmidt studied the sde in one dimensional case and formualted 
necessary and sufficient conditions to prove the existence of unique weak solution 
(see \cite{chen} for  details).

The existence of a unique strong solution is known when the coefficients are locally Lipschitz 
continuous with linear growth. If the drift coefficient is bounded, measurable and 
diffusion coefficient is  Lipschitz continuous and non-degenerate, then existence of  
a unique strong solution is known (see \cite{z} for the one dimensional case and
\cite{v} for the multidimensional case). Zvokin in \cite{z}, proved the result if the  
diffusion coefficient is Holder continuous  with exponent $\tfrac{1}{2}$.  

To the best of our 
knowledge, the above results are not known in the multidimensional case.
However there are some partial results. Swart in \cite{swart} proved the  pathwise uniqueness of an 
SDE evolving in unit ball in $\R^n$ with diffusion coefficient locally Lipschitz in
the  interior of the unit ball and Holder continuous with exponent
$\frac{1}{2}$ on the boundary. Fang and Zhang in \cite{fz} studied the
existence of  unique strong solution in the case when the
Lipschitzian asumption on the coefficients are relaxed by a
logarithmic factor. This still does not give results for the
Holder continuous class, analogous to the one dimensional case. 
The technique used in \cite{fz} does not seem to be working with 
Holder continuous coefficients. In this article we attempt to study the SDE with
Holder continuous coeeficients. Our approach is to use comparsion type arguments. 
To the best of our knowledge, this approach is not used to prove the existence of 
strong solutions in multi-dimensional case. We now briefly discuss the content of 
our article.
 
Consider the SDE
\begin{equation} \label{s1sde1}
\left\{
\begin{array}{lll}
d X(t)& = & \mu(X(t)) \, dt \, + \, \sigma(X(t)) \, d W(t)\\
X(0) & = & x \in \R^n_{+},
\end{array}
\right.
\end{equation}
where $\R^n_{+}  =  \{ x \in \R^n \, | \, x_i > 0 , \ \forall i \}$, 
the positive orthant. Using  comparison technique,
we obtain sufficient condition for any solution $X(\cdot)$ of (\ref{s1sde1})
to evolve in the positive orthrant, $\R^n_{+}$. As a corrollary, we obtain the
existence of a unique strong solution to (\ref{s1sde1}) under local Lipschitz
assumptions in $\R^n_{+}$ on the coefficients of the SDE. As an example we
establish the existence of a unique strong solution to the SDE 
\begin{equation} \label{s1sde2}
\left\{
\begin{array}{lll}
dX_i(t) & = & \mu_i(X(t)) \, dt \ + \ \sqrt{|X_i(t)|}
\sum^n_{i,j=1}
\sigma_{ij}(X(t)) \, d W_j(t) \\
X(0) & = & x \in \R^n_{+}, \ i =1, \cdots , n \, .
\end{array}
\right.
\end{equation}
The SDE (\ref{s1sde2}) is studied in \cite{siva,basp} in connection with super-Markov
chains. These authors have  proved the existence of a unique weak solution and left the
existence of a unique strong solution as an open problem. In this article, we are able 
to give an answer to this problem.

The equation (\ref{s1sde2}) can also be seen as a multidimensional version of CIR model 
in mathematical finance. The solution of sde (\ref{s1sde1}) or in particular (\ref{s1sde2}) remaining in
positive orthrant have a great significance in finance. In \cite{delhir}, the authors studied 
the positivity of a class of one dimensional
SDEs and obtain the nonarbitarge nature of the SDE. Typically 
the price movements of a collection of assets
(particularly equities) are modeled by certain SDEs. The
solutions of such SDEs must evolve in the positive orthant.
Indeed if the asset prices are allowed to hit zero, the model
creates arbitrage possibility. It may be emphasized  that in the
mathematical finance literature, effectively there is only one
class of models which captures the above positivity phenomenon,
viz.,the multidimensional version of the Black-Scholes model and
its nonlinear versions. Note that for modeling short rates, there
exists a variant of Bessel processes, the so called CIR model for
short term interest rates. Multidimentional version of CIR model
has not been studied in the mathematical finance literature. A
possible reason may be the non-availability of a sufficient
condition that ensures the evolution of the process given by the
multidimentional CIR model in the positive orthant.

We also show that our comparison technique can be used for other classes of SDEs by applying
to an example. More precisely we consider the SDE
\begin{equation}\label{s1sde3}
\left\{
\begin{array}{lll}
d \, X(t) & = & c (\theta - X(t)) dt \, + \, \sqrt{2 (1- \|X(t)\|^2)} \, d W(t) \\
X(0) & = & x \in B(0, \, 1) \subseteq \R^n \, ,
 \end{array}
\right.
\end{equation}
where $B(0, 1) \ = \ \{ x \in \R^n \, : \, \|x \| < 1\}$. 
We show that under the assumption $c(1- \sqrt{ n} \, |\theta | ) \geq 2$, the SDE (\ref{s1sde3})
has a unique strong solution evolving in $B(0, \, 1)$. The SDE (\ref{s1sde3}) is studied by Swart in
\cite{swart} for the case $\theta = 0$. The method used in \cite{swart} doesn't seems to work for 
$\theta \neq 0$.

Our paper is organized as follows: In Section 2, we prove the
existence of a unique strong solution to the SDE \eqref{sde1}
under the assumption (A1).   Section 3 contains a study of SDE 
in positive orthrant. We apply our main result in Section 3 to study
a multidimensional variant of Bessel process in Section 4. We apply our
technique developed in Section 3 to study an SDE in unit Ball in Section 5. 
We conclude our article with some conclusions.

\section{An Auxialary Result}
Consider the (possibly) degenerate SDE 
\begin{equation}
\begin{cases}
dX(t)= \mu(X(t)) dt + \sigma(X(t) \, dW(t) \\[2mm]  
X(0)  =  x \in \R^n 
\end{cases}\label{sde1}
\end{equation}
\noindent where $\mu: \R^n \rightarrow \R^n, \ \sigma : \R^n
\rightarrow \R^{n \times n}$ and $W(\cdot)$ is an $\R^n$-valued Wiener process.
In this section, we study the existence and uniqueness of a strong solution 
to the SDE \eqref{sde1}. We prove this under the following assumptions: 

\begin{itemize}

\item[{\bf{(A1)(i)}}] The function $\mu$ is Lipschtiz continuous on
compact subsets of $\R^n$.

\item[{\bf{(ii)}}] There exists $\epsilon > 0$,  a strictly increasing
function $\rho : [0, \ \infty) \to \R$ satisfying $\rho (0) \ = \
0$ and $ \int^t_0 \frac{1}{{\rho}^2 (s)} \, ds  \ = \ \infty$
such that
\[
| \sigma_{ij} (x) \, - \, \sigma_{ij}(y) | \ \leq \ C_R \, \rho
(|x_i - y_i |), 
\]
for some $C_R > 0$ and for all $x, y \in B(0, R)$ such that 
$\|x-y \| \leq \epsilon$.

\item[{\bf{(iii)}}] The functions $\mu$ and $\sigma$ have linear growth.

\end{itemize}

Following [\cite{nobu}, p. 182-183], we introduce some notation. Define
$\phi_k:\R\rightarrow\R$
as follows:

Let $1= s_0 > s_1 > s_2 > \cdots > s_k > \cdots > 0$ be such that
\[
 \int^{s_{k-1}}_{s_k} \frac{1}{\rho^2 (t)} \, dt
\ = \ k , k \geq 1\, .
\]
Let $\psi_k$ be continuous functions with supp$(\psi_k) \subseteq
(s_k , \ s_{k-1}) , k  \geq 1$, and
\[
0 \leq \psi_k (s) \leq \frac{2}{\rho^2(s) k} , \
\int^{s_{k-1}}_{s_k} \psi_k (s) \, ds\, = \, 1.
\]
Set
\[
\phi_k (t) \ = \  \int^{|t|}_0 \int^s_0 \psi_k
(\theta) \, d \theta \,
d s  \ \forall t \in \R
\]
Then $\phi_k$ satisfies
\begin{itemize}
\item[(i)] $ \phi_k \in C^2(\R)$
\item[(ii)] $0\leq \phi_k^{'}(t)\leq 1 $ \item[(iii)] $0\leq
\phi_k^{''}(t)\leq \frac{2}{k\rho^2(|t|)},~~~t\in \R , t \neq 0$
\item[(iv)]$\phi_k(t)~ \uparrow~ |t|$ as $k\rightarrow \infty.$
\end{itemize}
Now define $\bar{\phi}_k : \R^n \to \R $  by
\[
\bar{\phi}_k (x) \ = \ \sum^n_{i=1} \phi_k (x_i) \, .
\]
Then $\bar{\phi}_k$ has the following properties induced from
$\phi_k$.
\begin{itemize}
\item[(i)] $\bar{\phi}_k\in C^2(\R^n)$
\item[(ii)] $\nabla
\hat{\phi}_k(x) \ = \ (\phi_k^{'}(x_1),\cdots, \phi^{'}(x_n))$ ,
 \[
D^2\bar{\phi}_k(x) \ = \ \left[\begin{array}{llll}
\phi_k^{''}(x_1) & 0& \cdots& 0\\
0 & \phi^{''}_k (x_2)& \cdots& 0 \\
 & & \cdots & \\
 0& 0& \cdots & \phi^{''}_k(x_n)
\end{array}\right]
\]
\item[(iii)] $\bar{\phi}_k(x)~~ \uparrow ~~   \|x \| \, = \, |x_1|
+ \cdots + |x_n| $ as $k\rightarrow \infty$.
\end{itemize}

Our first result in this section is the following pathwise uniqueness result.

\begin{lemma}
Assume (A1). Then pathwise uniqueness  holds for \eqref{sde1}.

\end{lemma}

\begin{proof}

Let $X_i(\cdot), \, i =1,2$ be two solutions of (\ref{sde1}) with a
 prescribed Wiener process $W(\cdot)$. Set
\[
X(t) \ = \ X_1(t) - X_2(t).
\]
Note that
\[
dX(t) \ = \ (\mu (X_1(t)) \, - \, \mu(X_2(t))) \, dt \ + \ (\sigma
(X_1(t)) \, - \, \sigma(X_2(t)) \, dW(t) \,
\]
Using Ito's formula, we have
\begin{multline} \label{s2sde2}
d \bar{\phi}_k(X(t)) 
=  \sum^n_{i=1}(\mu_i(X_1(t))
\, - \, \mu_i(X_2(t))) \, \phi^{'}_k(X^i(t)) \, dt \\
{\hspace*{5mm}} + \frac{1}{2} \sum^n_{i=1}
\tilde{m}_{ij}(X_1(t), X_2(t)) \, \phi^{''}_k(X^i(t)) \, dt \\
 + \sum^n_{i,j=1} (\sigma_{ij}(X_1(t)) \, - \,
\sigma_{ij}(X_2(t))) \, \phi^{'}_k (X^i(t)) \, d W_j(t),
\end{multline}
where
\[
(\tilde{m}_{ij}(x, y)) \ = \  [\sigma(x) - \sigma(y)][\sigma(x) -
\sigma(y)]^{\perp} \, .
\]
 Define
\[
 \tau_R = \ \inf \{ t > 0 \ | \ |X_i(t)| \geq R \ {\rm for\ some} \ i \ {\rm or} \
 \|X(t) \| \geq \epsilon \},
\]
and
\[
 \tau = \ \inf \{ t > 0 \ | \ \|X(t) \| \geq \epsilon \},
\]
where $\epsilon > 0$ is from  (A1)(ii). Observe that $\tau_R \to \tau$ 
as $R \to \infty$, since linear growth
conditions of the coefficients of the (\ref{sde1}) guarantees the
nonexplosion of any solution of (\ref{sde1}), see [\cite{nobu}, Theorem 2.4, pp. 163-164]. From (\ref{s2sde2})
we have after taking expectation
\begin{multline} \label{s2eq1}
E \bar{\phi}_k(X(t \wedge \tau_R))  \\
{\hspace*{6mm}}= \ \sum^n_{i=1}E \int^t_0 (\mu_i(X(s \wedge \tau_R)) \, - \,
\mu_i(X(s \wedge \tau_R))) \,
\phi^{'}_k(X^i(s \wedge \tau_R)) \, ds \\
+ \frac{1}{2} \sum^n_{i=1} E \, \int^t_0
\tilde{m}_{ij}(X_1(s \wedge \tau_R), X_2(s \wedge \tau_R)) \,
\phi^{''}_k(X^i(s \wedge \tau_R)) \, ds
\end{multline}
Now
\begin{multline} \label{s2eq2}
 \Big| \sum^n_{i=1}E \int^t_0 (\mu_i(X(s \wedge
\tau_R)) \, - \, \mu_i(X(s \wedge \tau_R))) \, \phi^{'}_k(X^i(s
\wedge \tau_R)) \, ds
\Big|  \\
\leq \ \sum^n_{i=1}E \int^t_0 \Big| \mu_i(X(s \wedge
\tau_R)) \, - \, \mu_i(X(s \wedge \tau_R))\Big| \,
\Big|\phi^{'}_k(X^i(s \wedge \tau_R)) \Big| \, ds  \\
 \leq \ K_R \, \int^t_0 E \| X(s \wedge \tau_R) \| \, ds,
\end{multline}
where $K_R > 0$ is the Lipschitz constant of $\mu$ in $B(0, R)$.
Now
\begin{equation} \label{s2eq3}
\left\{
\begin{array}{ll}
\displaystyle{ \Big|\frac{1}{2} \sum^n_{i=1} E \, \int^t_0
\tilde{m}_{ij}(X_1(s \wedge \tau_R), X_2(s \wedge \tau_R)) \,
\phi^{''}_k(X^i(s \wedge \tau_R)) \, ds \Big|} & \\
\displaystyle{  ~~~~~~ \leq \ \frac{1}{2} \sum^n_{i=1} \int^t_0 n
\rho^2(| X^i(s \wedge \tau_R)|) \, \frac{2}{k \, \rho^2(|X^i(s
\wedge \tau_R)|) } \, ds }& \\
~~~~~ \leq \ \frac{n^2 \, t}{k}. &
\end{array}
\right.
\end{equation}
Substituting (\ref{s2eq2}), (\ref{s2eq3}) in (\ref{s2eq1}) and
letting $ k \to \infty$ we have
\[
E \|X(t \wedge \tau_R) \| \ \leq \ K_R \, \int^t_0 E \|X(s \wedge
\tau_R) \| \, ds \, .
\]
Using Gronwall's lemma we now have
\[
E \|X(t \wedge \tau_R ) \|\ = \ 0, \ {\rm for\ all}\ t \, .
\]
Letting  $R \to \infty$ and  using Fatou's lemma, we get
\begin{equation}\label{s2eq4}
X(t \wedge \tau) \ = \ 0 \ {\rm for\ all}\ t.
\end{equation}
To complete the proof, it is enough to show that $\tau = \infty$
a.s. If $P \{ \tau < \infty \} \, > \, 0$, then there exists $T>
0$ such that $P \{ \tau \leq T \} > 0. $ Hence  it follows from
(\ref{s2eq4}) that on $\{ \tau \leq T \}$, $X (\tau) = 0$ which is
false in view of the definition of $\tau$. Hence $P \{ \tau <
\infty \} \, = \, 0$. This completes the proof. 
\end{proof}

We now prove the main result of this section.

\begin{theorem} \label{exist} Under the Assumption (A1), the SDE (\ref{sde1})
has a unique strong solution for all time.
\end{theorem}

\begin{proof}
In view of Yamada-Watanabe theorem on pathwise uniqueness, \cite{yw}, [\cite{kara1}, pp.309-310],
to show the existence of a unique strong
solution, it is sufficient to show the pathwise uniqueness and existence of a 
weak solution. Lemma 2.1 guarantees the pathwise uniqueness. 
Since the drift 
and the diffusion coefficients are continuous and have linear growth,
the SDE (\ref{sde1}) has a weak solution 
for all time, see [\cite{nobu}, Theorem 2.3, Theorem 2.4, pp.159-164].  
This completes the proof of the theorem.  
\end{proof}

\begin{remark}
In the proof of  Lemma 2.1 (and hence in the proof of the existence of a unique 
strong solution to (2.1)), we extended the technique used for one dimension case
in \cite{yw} to the multidimentional case. This extension is based on the 
componentwise nature of the Assumption (A1)(ii). This technique will fail, if 
we replace the Assumption (A1)(ii)  with 
\[
 |\sigma_{ij}(x) \, - \, \sigma_{ij}(y) | \ \leq \ C_R \rho(\|x-y\|)
\]
for all $x,y \in B(0, \, R)$, with $\|x-y\| \leq \epsilon$.
\end{remark}

We now present an example satisfying the Assumption (A1).
\begin{example}
Consider the SDE
\[
\left\{
\begin{array}{lll}
d X_i(t) & = & \mu_i(X(t)) \, dt \ + \ |X_i(t)|^{\beta}
\sum^n_{j=1}
\sigma_{ij}(X_i(t)) d W_j(t) \\
X (0) & = & x \in \R^n, \ i=1,2, \cdots, n \, ,
\end{array}
\right.
\]
where $ \frac{1}{2} \leq \beta \leq 1$ and $\mu_i, \ \sigma_{ij}$
are Lipschtiz continuous.  Then by Theorem 2.1, the above SDE has
unique strong solution.
\end{example}

\section{SDE in Positive Orthrant}

In this section, we study the sde (\ref{sde1}) in the positive orthrant. Under certain
assumptions, we show that any solution starting from any point in the positive orthrant 
of the sde  remains in the positive orthrant for all time. 

We introduce some notation before stating the assumptions.

Let $\rho_1 : \R_{+} \to \R$ be a strictly increasing
function satisfying $\rho_1(0) = 0$ and $\int_{0 +}
\frac{1}{\rho^2_1(s)} \, ds \, = \, \infty$. 

Define $p_i:\R^n_+\rightarrow \R$,
  $a_i:\R^n_+\rightarrow
\R$ and $ b_i:\R^n_+\rightarrow\R, \ i =1, \cdots, n$ by
\begin{eqnarray*}
p_i(x)=  x_i, \  a_i(x) \, =  \,
m_{ii}(x),~~b_i(x)=\frac{\mu_i(x)}{a_i(x)} \, ,
\end{eqnarray*}
where $(m_{ij}(x))\ =\ \sigma(x) \sigma(x)^{\perp}$. Also define
\[
a^{+}_i(r)=\sup_{x:p_i(x)=r} a_i(x), \
b^{-}_i(r)=\inf_{x:p_i(x)=r}b_i(x).
\]

We are now ready to state our assumptions.  

\begin{itemize}
\item[{\bf (A2)}] The functions $\mu$, $\sigma$ are continuous
with linear growth. Also $a^{+}_i(r) < \infty$ for all $ r > 0$. \\

\item[{\bf(A3)(i)}]  The function $\mu_i(x)  >0$ for all $ x$
in a nbd of $\partial \R^n_{+}$, say ${\mathcal D}$ , $ i=1,2,\cdots,n.$

\item[{\bf (ii)}] The functions $b^{-}_i, \, i =1, 2, \cdots, n$ satisfies
 \[
 b^{-}_i(r) \ > \  \frac{1}{r} \  {\rm for\ all} \ r>0.
\]

\item[{\bf (iii)}] There exists a function $\tilde{\sigma} : [0, \
\infty) \to \R$ such that
\[
m_{ii}(x) \ \leq \ n^3 {\tilde{\sigma}}^2(p_i(x)) \ {\rm for\ all}\
x \in {\mathcal D} , \ i= 1, 2, \cdots, n \, ,
\]
where $\tilde{\sigma}$ satisfies the following :
there exists an $\epsilon > 0$ and a $C^1_R > 0$ which depends on $R$
such that for $r_1, \ r_2 \in [0, \ R]$ with $|r_1 - r_2| \leq \epsilon$
\[
\begin{array}{lll}
| {\tilde{\sigma}}^2 (r_1) - {\tilde{\sigma}}^2(r_2) | & \leq & C^1_R \,
| r_1 - r_2|  , \\
| \sqrt{r_1} \tilde{\sigma} (r_1) - \sqrt{r_2} \tilde{\sigma}(r_2)
| & \leq & C^1_R \, \rho_1(|r_1-r_2|) \, ,
\end{array}
\]
\end{itemize}

We are now ready to state and prove the main theorem of our article.

\begin{theorem}\label{The-nonnegative}
Assume (A2) and (A3). Let
$X(\cdot)=(X_1(\cdot),\cdots,X_n(\cdot))$ be a solution to
(\ref{sde1}) with the initial condition $X_i(0) = x_i
> 0, \ i=1,2,\cdots,n.$ Then with probability $1$, paths of
$X_i(\cdot)$ are positive, $i=1,2,\cdots,n$.
\end{theorem}

\begin{proof} Fix $i$. Consider the process
$\hat{X}(\cdot)$ satisfying
\[
d \hat{X}(t) \ = \ \mu(\hat{X}(t)) \, dt \ + \ \hat{\sigma}(\hat{X}(t))\, d \,
W(t) \, ,
\]
where $\hat{\sigma}$ satisfies
\[
\hat{\sigma}(x) \ = \
\left\{
\begin{array}{lll}
\sigma(x) & {\rm if}& x \in {\mathcal D} \\
\tilde{\sigma}(p_i(x))& {\rm if}& x \in {\mathcal D}^c_1
\end{array}
\right.
\]
for a  nbd ${\mathcal D}_1$ of $\partial {\R}^n_{+}$ such
that ${\mathcal D} \subseteq {\mathcal D}_1$ and $\hat{\sigma}$
satisfies
\[
\sum^n_{k=1} \hat{\sigma}_{ik}(x) \, \hat{\sigma}_{ki}(x) \ \leq \ n^3 \, \tilde{\sigma}^2(p_i(x))
\ {\rm for} \ x \in {\mathcal D}_1 \setminus {\mathcal D} \, .
\]
Such a  $\hat{\sigma}$ can be constructed as follows:
Using Urysohn's lemma, there exists a continuous function $f : \R^n \to [0, \ 1]$ such that
$f(x) = 0$ on $\overline{{\mathcal D}}$ and $ f(x) \, = \, 1$ on ${\mathcal D}^c_1$. Now choose
\[
 \hat{\sigma}(x) \ = \ (1- f(x)) \sigma(x) + f(x) \tilde{\sigma}(p_i(x)) , \ x \in \R^n
\]

Note that the process $\hat{X}(\cdot)$ given above and the process $X(\cdot)$
given by \eqref{sde1} behave identicaly in some nbd of $\partial {\R}^n_{+}$.
Hence without the loss of generality we can assume that $\sigma \, = \ \hat{\sigma}$.

Set 
\[
\tau=\inf\{t\geq 0\, | \, X(s)\in \partial \R_+^n\}
\] 
and define
\begin{equation*}
\phi(t)=\int_{0}^{t\wedge
\tau}\frac{a_i(X(s))}{a^{+}_i(p_i(X(s)))}ds,
\end{equation*}
Then $\phi$ is pathwise smooth, strictly increasing and
$\phi(0)=0$. Let $\psi=\phi^{-1}.$ Set $Y_i(t)=X_i(\psi(t))$.
Using the time change of Brownian motion arguments as in
[\cite{nobu}, pp. 183-185], we can show that
$Y(\cdot):=(Y_1(\cdot),\cdots,Y_n(\cdot))$ satisfies
\begin{equation}\label{s3sde1}
\left\{
\begin{array}{lll}
dY_i(t) & = & \displaystyle{
\frac{a^+_i(p_i(Y(t)))}{a_i(Y(t))} \mu_i(Y(t))dt } \\
& & \displaystyle{
+\sqrt{\frac{a^+_i(p_i(Y(t)))}{a_i(Y(t))}}
\sum_{j=1}^{n}\sigma_{ij}(Y(t))d\bar{W}_j(t)}
\end{array}
\right.
\end{equation}
for some Brownian motion $\bar{W}(\cdot)$.\par

Using  Ito's formula to $p_i(\cdot)$ and the
process (\ref{s3sde1}), we obtain
\[
\left\{
\begin{array}{lll}
dp_i(Y(t)) & = & \displaystyle{
\frac{a^+_i(p_i(Y(t)))}{a_i(Y(t))}
\mu_i(Y(t))dt} \\
& & \displaystyle {
+\sqrt{\frac{a^+_i(p_i(Y(t)))}{a_i(Y(t))}} \sum_{j=1}^{n}
\sigma_{ij}(Y(t))d\bar{W}_j(t) }
\end{array}
\right.
\]
i.e.,
\begin{equation}\label{s3sed2}
dp_i(Y(t))=a^+_i(p_i(Y(t)))b(Y(t))dt+\sqrt{a^+_i(p_i(Y(t)))}d\tilde{W}(t),
\end{equation}
for some one-dimensional Brownian motion $\tilde{W}(\cdot)$.\\

 Now consider the following SDE
\begin{equation}\label{s3sde3}
dZ(t)=\frac{a^+_i(Z(t))}{Z(t)}dt+\sqrt{a^+_i(Z(t))}d\tilde{W}(t).
\end{equation}
By applying Ito's formula we can see that $Z(\cdot)$ is a solution
to \eqref{s3sde3} iff $Z_1(\cdot)$ is a solution to
\begin{equation}\label{s3sde4}
dZ_1(t)= 3 \, a^+_i(\sqrt{Z_1(t)}) \, dt \, + \, 2
\sqrt{Z_1(t)}\sqrt{a^+_i(\sqrt{Z_1(t)})} \, d\tilde{W}(t).
\end{equation}
Using (A3)(iii),
we have
\[
a^{+}_i (r) \ = \ n^3 {\tilde{\sigma}}^2(r) , \ r > 0 \, .
\]
Observe that the equation \eqref{s3sde4} satisfies all the assumptions in
the Theorem 2.1. Thus \eqref{s3sde4} has a unique strong solution. As a 
consequence the equation \eqref{s3sde3} will also have a unique strong
solution for all time.

Now take
\[
b^1(r) \ = \ a^{+}_i(r) \, b^{-}(r) , \ b^2(r) \ = \
\frac{a^{+}(r)}{r} \, .
\]
Then
\[
b^1(r) \ > \ b^2(r) , \ {\rm for\ all}\ r > 0
\]
and
\[
a^{+}(p_i(Y(t)) \, b(Y(t)) \, \geq \ b^{1}(p_i(Y(t))) \, .
\]
Here note that $b^1(0) $ may or may not be equal to $b^2(0)$. If
$b^1(0) > b^2(0)$, then a direct application of
 comparison theorem, see [ \cite{nobu}, p.437], we have
\begin{equation} \label{s3esti1}
Z(t)\ \leq \ p_i(Y(t)) \ \mbox{a.s.}
\end{equation}
If $b^1(0) \ = \ b^2(0)$, then apply comparison theorem up to the
stopping time $\eta_n \ = \ \inf \{\,  t > 0 \ | \ p_i(Y(t) \leq
\frac{1}{n} \ {\rm or}\ Z(t) \leq \frac{1}{n} \, \}$ and let $n
\to \infty$, we have (\ref{s3esti1}). \\

Set $M(t)=\log Z(t).$ Then Ito's
formula implies that
\begin{equation*}
dM(t)=\frac{1}{Z(t)}\sqrt{a^+(Z(t))}d\tilde{W}(t)
\end{equation*}
i.e. for $\eta =\sup\{s\geq0||M(s)|<\infty\}$, $M(t)=\log \, Z(t)$
defines a local Martingale in $[0,\eta)$. Now using Corollary
34.13 of [\cite{row}, p.67],  on $\{\eta<\infty\}$ we have
\begin{equation*}
\overline{\lim}_{t\uparrow\eta} M(t)=\infty
\end{equation*}
i.e.
\begin{equation*}
\overline{\lim}_{t\uparrow\eta} Z(t)=\infty.
\end{equation*}
Now since (\ref{s3sde3}) has a unique strong solution for all $ t
\geq 0$, we have $\{\eta<\infty\}$ has probability $0$. That is
$Z(t)>0$ a.s., for all $t\geq 0$. Hence from (\ref{s3esti1}) it
follows that $Y_i(\cdot) > 0 $ a.s.
\end{proof}

\section{Multidimensional Variant of Bessel Process}

In this section we prove the existence of unique strong solution
to the SDE 
\begin{equation} \label{s4sde1}
\left\{
\begin{array}{lll}
dX_i(t) & = & \mu_i(X(t)) \, dt \ + \ \sqrt{|X_i(t)|}
\sum^n_{j=1}
\sigma_{ij}(X(t)) \, d W_j(t) \\
X(0) & = & x \in \R^n_{+},
\end{array}
\right.
\end{equation}
The SDE (\ref{s4sde1}) can be seen as  the multidimensional variant of the Bessel process
given by
\[
 d X(t)  \ = \ c \, dt \, + \, 2 \sqrt{|X(t)|} d W(t), \ X(0) \, = \, x \geq 0 \, .
\]
Positive Diffusion processes play a central role in Mathematical finance due to
the arbitragefree nature of it, see for example \cite{delhir}. 
The equation (\ref{s4sde1}) is studied in \cite{siva, basp} in connection with
super-Markov chains. Pathwise uniqueness is posed as an open problem in \cite{siva}. 
Pathwise uniquness results for the SDE with non-Lipschtiz coefficients are very limited
in the multidimensional case. Fang and Zhang in \cite{fz} provided 
pathwise uniqueness results for a general class of SDEs.  But these resutls can not 
be applied for SDEs whose coefficients are  Holder continuous. Thus \eqref{s4sde1} can not
be studied using the results in \cite{fz}. Our object in this section is to study this equation.
We first prove a general theorem.

\begin{theorem} Assume that $\mu$ is locally Lipschitz in $\R^n_{+}$ with linear
growth and satisfies (A3)(i), (ii). Also assume that $\sigma$ is
locally Lipschitz in $\R^n_{+}$ with linear growth and satisfies (A3)(iii). Then the
SDE (\ref{sde1}) has a unique strong positive solution.
\end{theorem}

\begin{proof} Since the coefficients of the SDE
(\ref{sde1}) are locally Lipschtiz continuous with linear
growth, using classical result of Ito, existence of unique weak
solution for all time and pathwise uniqueness up to hitting time $\tau \ = \
\inf \{ t > 0 \, | \, X(t) \in \partial \R^n_{+} \}$ follows. Now
using Theorem 3.1, we have $\tau \, = \, \infty$ a.s. Hence the
SDE (\ref{sde1}) has pathwise uniqueness for all time. Hence
using Yamada-Watanabe theorem \cite{yw}, we can conclude the existence of
unique strong solution to (\ref{sde1}) for all time. 
\end{proof}

We now study \eqref{s4sde1} and prove the existence and uniquenss in the following
theorem.

\begin{theorem} Assume that $\mu$ is locally Lipschitz with linear
growth and satisfies (A3)(i), (ii). Also assume that $\sigma$ is
bounded and locally Lipschitz, then the SDE (\ref{s4sde1}) has a
unique positive strong solution.
\end{theorem}

\begin{proof} Inview of Theorem 4.1, it only remains to show that
$\sigma$ satisfies (A3)(iii).
Since  $\sigma$ is bounded, locally  Lipschitz
continuous, we have
\[
\begin{array}{lll}
m_{ii}(x) & = & \displaystyle{ \sum^n_{j=1} x_i  \, \sigma^2_{ij}(x) }\\
     & \leq & n \, K \, x_i \ = \ n^3 \frac{K}{n^2} p_i(x),
 \end{array}
\]
where $K > 0$ is a bound for $\sigma$. Hence $\sigma$ satisfies
(A3)(iii).  
\end{proof}

\section{SDE in Unit Ball}

Our aim in this section is to show that the method developed in Section 3 
can be adapted to study equations evolving in domains other than positive 
orthrant. We restrict our attention to a variant of an SDE in unit ball 
studied by Swart in \cite{swart}.

Consider the SDE
\begin{equation}\label{s4sde2}
\left\{
\begin{array}{lll}
d \, X(t) & = & c (\theta - X(t)) dt \, + \, \sqrt{2 (1- \|X(t)\|^2)} \, d W(t) \\
X(0) & = & x \in B(0, 1) \subseteq \R^n \, .
 \end{array}
\right.
\end{equation}
The SDE (\ref{s4sde2}) is studied in \cite{swart} with $\theta = 0$. He showed the 
existence of a unique strong solution for $\theta =0$ and $ c \geq 1$. 

We study the equation \eqref{s4sde2} for any $\theta$ with the assumption that\\
$c ( 1 - \sqrt{n} \, |\theta |) \geq 2$.

We now prove our result in the following theorem.

\begin{theorem}
 Assume   $c ( 1 - \sqrt{ n} \, |\theta |) \geq 2$. 
Then the SDE (\ref{s4sde2}) has unique strong solution. More over solution evolves in
$B(0, 1)$. 
\end{theorem}

\begin{proof} In view of the arguments in the proof of Theorem 4.1, it suffices to show that
any solution to (\ref{s4sde2}) does not hit the boundary of $B(0, 1)$.

Let $X(\cdot)$ be a solution to (\ref{s4sde2}) corresponding to the Wiener process $W(\cdot)$. 
We can assume without the loss of generality that $ X(0) \, = \, x \neq 0$. Define
\[
 p(x) \, = \, \|x\|^2,  \ a(x) \, = \, 8 \, p(x) \, (1 - p(x))
\]
and
\[
b(x) \, = \, \frac{2 [n - (n+c) p(x)] \, + \, 2 c \, \theta \sum^n_{i=1} x_i}{a(x)}, \ x \in \R^n \, .
\]

Set
\[
 a^{+}(r) \ = \ \sup_{p(x) = r} a(x), \ b^{+}(r) \ = \ \sup_{p(x) = r} b(x) , \ r > 0 \, .
\]
Then
\[
 a^{+}(r) \ = \ 8 r (1-r), \ b^{+}(r) \ =  \ \frac{2 [n -(n+c(1- \sqrt{n}\, |\theta| ))r]}{a^{+}(r)} \, .
\]
Let $ \tau \ = \ \inf \{ t \geq 0 \, | \, \|X(t)\| \geq 1 \} $ and
\[
 \phi (t) \ = \ \int^{ t \wedge \tau}_0 \frac{ a (X(s))}{a^{+}(p(X(s))} \, ds \, .
\]
Now mimicking the argumants from Theorem 3.1, there exists one dimensional Wiener 
process $\tilde{W}(\cdot)$ such that $Y(\cdot)$ given by $Y(t) \ = \ X(\psi(t)), \ 
\psi  \ = \ \phi^{-1}$ satisfies 

\begin{equation}\label{s4sde3}
\left\{
\begin{array}{lll} 
d p(Y(t)) & = & a^{+}(p(Y(t)))\,  b(Y(t))  dt  +  \sqrt{8 \, p(Y(t)) (1 - p(Y(t)))} d \tilde{W}(t)\\
p(Y(0)) & = & p(x) \, ,
\end{array} \right.
\end{equation}

Consider the SDE 
\begin{equation}\label{s4sde4}
\left\{
\begin{array}{lll} 
d Z(t) & = & 
2[ n - (n+c(1- \sqrt{n} \,  |\theta |)) Z(t)]  dt  +  \sqrt{8 Z(t) (1 - Z(t))} 
d \tilde{W}(t)\\
Z(0) & = & p(x) \in (0, \ 1) \, ,
\end{array} \right.
\end{equation}
Since $\sqrt{ 8 x (1-x)}$ is Holder continuous with exponent $\frac{1}{2}$, the SDE (\ref{s4sde4})
has unique strong solution. Also note that for $n \geq 2$
\[
\begin{array}{lll}
\displaystyle{ \int^x_{\frac{1}{2}} e^{- 2 \, \int^y_{\frac{1}{2}} 
\frac{2[n - (n+c(1- \sqrt{n} |\theta |))u]}{8 u(1-u)}\, du}
\, dy } & = & \displaystyle{2^{\frac{c(\sqrt{n} |\theta| -1) - n}{2}} \, 
\int^x_{\frac{1}{2}} y^{- \frac{n}{2}} \, (1-y)^{- \frac{c(1- \sqrt{n} |\theta |)}{2}} \, dy }\\
& = & \infty
\end{array}
\]
By [\cite{kara1}, Proposition 5.22, p. 345], we have $0 < Z(t) < 1$ a.s. for $t$. Now note that
\[
 a^{+}(p(y)) b(y) \ \leq \ a^{+}(p(y)) \, b^{+}(p(y)) \ = \ 2 [n- (n+ c(1- \sqrt{ n} |\theta |)) p(y)] \, . 
\]
Hence using comparison theorem, [ \cite{nobu}, p. 437], we have
\begin{equation}\label{s5eq1}
 p(Y(t) ) \ \leq \ Z(t) \ < 1 \ {\rm a.s}\ \forall t \, 
\end{equation}
i.e $\|Y(t)\| < 1$ a.s. for all $t$. Hence the process $X(\cdot)$ never hits the boundary of
$B(0, \, 1)$. 

\end{proof}

\begin{remark} From (\ref{s5eq1}),  the process $X(\cdot)$ given by SDE (\ref{s4sde2}) 
can be in boundary only when the process $Z(\cdot)$ given by (\ref{s4sde4}) is at $1$.
Hence by looking at the behaviour of $Z(\cdot)$ at $1$, we can get a sufficient condition for
the process $X(\cdot)$ spending zero time on the boundary. Now we can apply the argument of
\cite{swart} to show the pathwise uniqueness of (\ref{s4sde2}) under the case when (\ref{s4sde2})
spends zero time on the boundary.

\end{remark}

\section{Conclusion}

In this article, we prove the existence of a unique strong
solution for a class of multidimensional SDEs with non-Lipschitz
diffusion coefficients. The analogous result for the
one-dimensional case was known for a long time but the
multidimensional version wasn't available. We prove the result
under the assumptions (A2) and (A3). The proofs of our results
are based on  comparison theorem arguments. We believe that it is
for the first time  that comparison theorems are exploited to
derive such results. \\

\noindent {\it Acknowledgements.} I would like thank Gopal K. Basak
for pointing out an error in a preliminary draft of this paper. I
also wish to thank  Mrinal K. Ghosh and K. S. M. Rao for giving me useful
suggestions to improve the writeup.

\bibliography{ref}
\bibliographystyle{plane}

\end{document}